\documentclass[11pt]{amsart}

\usepackage{amsmath,amsfonts,amsthm,amssymb}
\usepackage[margin=1in,letterpaper]{geometry}
\usepackage{hyperref}

\usepackage{physics}
\usepackage{graphicx}
\usepackage{tikz}
\usepackage{tikz-cd}
\usetikzlibrary{arrows}
\usepackage{amsthm}
\theoremstyle{definition}
\newtheorem{defi}{Definition}[section]
\newtheorem{theorem}[defi]{Theorem}
\newtheorem{corollary}[defi]{Corollary}
\newtheorem{lemma}[defi]{Lemma}
\newtheorem{prop}[defi]{Proposition}

\newtheorem{remark}[defi]{Remark}

\usepackage{bm}
\usepackage{bbm}

\DeclareMathOperator{\diag}{diag}

\DeclareMathOperator{\sgn}{sgn}
\DeclareMathOperator{\Hilb}{Hilb}
\DeclareMathOperator{\contract}{contract}  
\providecommand{\keywords}[1]{\textbf{\textit{Keywords---}} #1}

\title{Parseval-Rayleigh Identities for Homogeneous Complete Intersections}
\author[Karim Adiprasito]{Karim Alexander Adiprasito}
\address{{Karim Adiprasito,\ Ryoshun Oba \ \emph{and}\ Vasiliki Petrotou}, Sorbonne Université 
and Université Paris Cité, CNRS, IMJ-PRG, F-75005 Paris, France}
\email{adiprasito@imj-prg.fr, oba@imj-prg.fr \emph{and} petrotou@imj-prg.fr}

\author{Ryoshun Oba}

\author[Stavros Papadakis]{Stavros Argyrios Papadakis}
\address{{Stavros Papadakis}, Department of Mathematics, University of Ioannina, Ioannina, 45110, Greece}
\email{spapadak@uoi.gr}
\author[Vasiliki Petrotou]{Vasiliki Petrotou}
\date{}


\begin{document}  

\begin{abstract}
We prove, in any positive characteristic, Parseval-Rayleigh identities for the residue map of a
homogeneous complete intersection.
As an application, we give a conceptual proof of the folklore fact 
that  generic homogeneous complete 
intersections have the Strong Lefschetz Property over any field of characteristic $2$.
\end{abstract}

\maketitle

\section{Introduction} 

Parseval-Rayleigh identities originate in Fourier analysis.
An algebraic analogue in positive characteristics has recently been introduced 
in the context of understanding Lefschetz properties for semigroup algebras of lattice polytopes and face rings. 
In the following, 
by Parseval-Reyleigh identities we will only refer to the algebraic analogue. 
The Parseval-Reyleigh  identities  have been the key ingredient in  the resolution of several 
long-standing combinatorial  conjectures for $h^\ast$-vectors \cite{APP}.
They appear in \cite {KLS} in the context of Stanley-Reisner rings of simplicial pseudomanifolds.
They also explain earlier observed phenomena of generic Lefschetz properties \cite{AHL, PP,  KX, APPHAL}.
Hence they  seem to be of considerable importance, but
their existence has been established in very few cases. 

In the present note, we establish  Parseval-Rayleigh identities for  homogeneous complete intersections
in any positive characteristic.

\subsection{Statement} \label{subsec:Parseval statement} 
For the statement, we assume that ${\mathbbm{k}}$ is a field of positive characteristic $p$.
Let $R={\mathbbm{k}}[x_1,\ldots,x_m]$ be a polynomial ring with the standard grading, and let $g_1,\ldots,g_m$ be 
a homogeneous $R$-regular sequence of length $m$ defining an Artinian complete intersection ideal $I=(g_1,\ldots,g_m) \subset R$.
We set $ s= - m + \sum_{i=1}^m \deg g_i$, by Section~\ref{section:normalization}  $s$ is the top degree  of $R/I$.
Let $\pi:R \rightarrow R/I$ be the natural projection, and let $\mathrm{vol}:(R/I)_s \rightarrow {\mathbbm{k}}$ be 
the residue map, see Section~\ref{section:normalization}.

For $i \geq 0$, we denote by $\contract_i:R_i \times R_i \rightarrow {\mathbbm{k}}$ the symmetric bilinear form which is uniquely 
specified by the condition $\contract_i (u,u')=\delta_{u,u'}$ for any two monic monomials $u,u' \in R_i$.
Here  $\delta$ denotes the Kronecker delta.
We denote $\text{contract}_i(u,w)$ simply as $u\circ w$.  
We have the following identities, which we call    \emph{Parseval-Rayleigh identities}.
\begin{theorem}[\bf{Parseval-Rayleigh identities}] \label{thm:Parseval}
    For any $w \in R_s$ we have
     \begin{align} \label{eq:PI}
         \mathrm{vol}(\pi(w)) = \sum_{u \in {\mathcal M}_s} ( (x_1^{p-1}\cdots x_m^{p-1} u^p)
         \circ (g_1^{p-1}\cdots g_m^{p-1}w)) \; (\mathrm{vol}(\pi(u)))^p,
     \end{align}
    where ${\mathcal M}_s$ denotes the set of all degree $s$  monic monomials of $R$.
\end{theorem}  

In Section~\ref{sec!dislp} we 
discuss  applications, such as anisotropy properties for the Hodge-Riemann pairing for complete intersections, 
generalizing well-known Lefschetz properties \cite{MN, HMMNWW, HMNW, Stanley}.
In particular, in Corollary~\ref{cor:Lefschetz}  we give a  proof of the folklore fact 
that over any  (finite or infinite)  field of characteristic $2$
generic homogeneous complete  intersections have the Strong Lefschetz Property, where
generic  here 
means that the coefficients are new independent variables.    \\

\paragraph{{\it Acknowledgements.}}  S.~A.~P. and V.~P. thank Frank-Olaf Schreyer for useful discussions. 
 K.~A.~A., R.~O. and V.~P. were supported by Horizon Europe 
ERC Grant number: 101045750 / Project acronym: HodgeGeoComb. R.~O. was supported by JSPS Overseas Research Fellowships. 
We benefitted from experimentation with Macaulay2 \cite{M2}.

\section{Notation}   \label{section:normalization}
Let ${\mathbbm{k}}$ be a field.
Let $R={\mathbbm{k}}[x_1,\ldots,x_m]$ be a polynomial ring with the standard grading $\deg x_i=1$ for $i=1,\ldots,m$.
Denote by $R_i$ the degree $i$ homogeneous component of $R$ and by $R_+$ the irrelevant ideal $
(x_1,\ldots,x_m)$ of $R$,  that is $R_+=\oplus_{i \geq 1} R_i$.
Suppose that $g_1,\ldots,g_m \in R_+$ is a homogeneous $R$-regular sequence of length $m$. 
It defines an Artinian complete intersection ideal $I=(g_1,\ldots,g_m)$ of $R$, and the Artinian Gorenstein ring $R/I$.
Let $\pi:R \rightarrow R/I$ be the natural projection map.
The Hilbert series $\Hilb(R/I, t)$ of $R/I$ is computed as
\[
\Hilb(R/I, t) = \Hilb(R, t) \prod_{i=1}^m (1-t^{\deg g_i}) = \prod_{i=1}^m \sum_{j=0}^{\deg g_i -1} t^j.
\]
Hence the top degree (also known as socle degree) of $R/I$ is
\begin{equation}
    s := -m  + \sum_{i=1}^m \deg g_i,
\end{equation}
and we have $\dim_{\mathbbm{k}} (R/I)_s =1$.

We denote by 
\[
           \mathrm{vol} : (R/I)_s \to {\mathbbm{k}}
\]           
 the  homogeneous residue isomorphism discussed in \cite[Section~1.5.6]{Cattani2005}.
It can be defined as follows:

We fix an $m \times m$ matrix $N$  with entries  in $R$ satisfying
\begin{equation} \label{eq:coeffmatrix}
\begin{pmatrix}
    g_1 \\ \vdots \\ g_m
\end{pmatrix}
=
N
\begin{pmatrix}
    x_1 \\ \vdots \\ x_m
\end{pmatrix},
\end{equation}
and we set 
\begin{equation} \label{eq:z_0}
z_0=\det N.
\end{equation}
Note that $\deg z_0=s$.   The following proposition is  well known, 
see for example    \cite[Lemma~E.19]{Kunz}.

\begin{prop} \label{prop:well-def}
 The element     $\pi(z_0)$  of    $R/I$ is independent of the choice of the 
 matrix $N$ in Equation~(\ref{eq:coeffmatrix}). Moreover,   $\pi(z_0) \not= 0$.
\end{prop}

By definition,  $\mathrm{vol}:(R/I)_s \rightarrow {\mathbbm{k}}$ 
is the unique  linear isomorphism 
of $1$-dimensional  ${\mathbbm{k}}$-vector spaces   satisfying 
\begin{equation} \label{eq:normalization}
         \mathrm{vol}(\pi(z_0))=1.
\end{equation}

\begin{remark} 
It is not relevant to the present paper, but we remark that, in general, 
the isomorphism $\mathrm{vol}$ depends on both the choice of the homogeneous 
generating set $g_1,\ldots,g_m$ of $I$ and its ordering. 
\end{remark}

\begin{remark} 
Our point of view  is that,    as discussed in 
\cite[Subsection~5.1]{APP2},  the  linear isomorphism
 $\mathrm{vol}$ is a case of  Kustin-Miller normalization of volume maps
\cite{APP2}. This is the reason
we are making the choice to use the notation  $\mathrm{vol}$ instead of 
 $ \mathrm{res}$. 
\end{remark}

\section{Parseval-Rayleigh Identities}

The main aim of the present section is the proof of Theorem~\ref{thm:Parseval}.

\subsection{A key ideal membership}
We assume that we are in the same setting as in Section~\ref{subsec:Parseval statement}.
The following is a key ingredient in the proof of Theorem~\ref{thm:Parseval}.
\begin{lemma} \label{lem:member}  We have
  \[
      g_1^{p-1}\cdots g_m^{p-1}z_0 - x_1^{p-1}\cdots x_m^{p-1} z_0^p  \;  \in (g_1^p,\ldots,g_m^p),
 \]
 where $z_0$ was defined in Equation~(\ref{eq:z_0}).
\end{lemma}
\begin{proof}
The idea of the proof is as follows.
Think of Lemma~\ref{lem:member} as an equality in the quotient ring $B=R/(g_1^p,\ldots,g_m^p)$.
Since, by   \cite[Exercise~1.1.10]{BH}, $g_1^p,\ldots,g_m^p$ is a regular sequence 
of $R$, $(g_1^p,\ldots,g_m^p)$ is an Artinian complete intersection ideal.
The top degree of $B$ is 
\[
       p(\sum_{i=1}^m \deg g_i)-m =ps+(p-1)m.
\]       
This
is equal to the degrees of both $g_1^{p-1}\cdots g_m^{p-1}z_0$  and  $x_1^{p-1}\cdots x_m^{p-1} z_0^p$.
To deduce that two values in the top degree are equal, we use Proposition~\ref{prop:well-def} for $B$.
We shall now give the argument.

We write $g_1^p,\ldots,g_m^p$ as an $R$-coefficient sum of $x_1,\ldots,x_m$ in two different ways.
The first way is
\begin{align} \label{eq:rep1}
\begin{pmatrix}
    g_1^p \\ \vdots \\g_m^p 
\end{pmatrix}
=
\diag(g_1^{p-1},\ldots,g_m^{p-1})
\begin{pmatrix}
    g_1 \\ \vdots \\g_m 
\end{pmatrix}
=
\diag(g_1^{p-1},\ldots,g_m^{p-1}) N
\begin{pmatrix}
    x_1 \\ \vdots \\x_m 
\end{pmatrix}.
\end{align}
To get the second way, since we are working over a field of characteristic $p$ it holds that
\[
    g_i^p =(n_{i1}x_1+\cdots +n_{im}x_m)^p = n_{i1}^p x_1^p +\cdots +n_{im}^px_m^p,
\]
where $n_{ij}$ is the $(i,j)$-entry of $N$. Hence, we have
\begin{align} \label{eq:repa}
\begin{pmatrix}
    g_1^p \\ \vdots \\g_m^p 
\end{pmatrix}
=
N^{(p)}
\begin{pmatrix}
    x_1^p \\ \vdots \\x_m^p 
\end{pmatrix}
=
N^{(p)}
\diag(x_1^{p-1},\ldots,x_m^{p-1})
\begin{pmatrix}
    x_1 \\ \vdots \\x_m 
\end{pmatrix},
\end{align}
where $N^{(p)}$ is the $m \times m$  matrix with  $(i,j)$-entry equal to  $n_{ij}^p$.
By Proposition~\ref{prop:well-def} for $B$ and $g_1^p, \ldots, g_m^p$, we deduce that
\begin{equation}
\pi_B\left(g_1^{p-1}\cdots g_m^{p-1}\;  (\det N) \right) = \pi_B \left((\det N^{(p)}) \;  x_1^{p-1} \cdots x_m^{p-1}\right),
\end{equation}
where $\pi_B : R \rightarrow B$ is the natural projection. To complete the proof, 
it remains to check that $\det N^{(p)}=(\det N)^p$, which  can be done as follows:
\begin{align*}
    \det N^{(p)} = \sum_{\sigma \in {\mathfrak {S}}_m} \sgn(\sigma) n_{1\sigma(1)}^p\cdots n_{m\sigma(m)}^p
    &=\sum_{\sigma \in {\mathfrak {S}}_m} \sgn(\sigma)^p n_{1\sigma(1)}^p\cdots n_{m\sigma(m)}^p \\
    &=\left(\sum_{\sigma \in {\mathfrak {S}}_m} \sgn(\sigma) n_{1\sigma(1)}\cdots n_{m\sigma(m)}\right)^p
    =(\det N)^p,
\end{align*}
where $\mathfrak{S}_m$ is the group of permutations of $\{1,\ldots,m\}$.
\end{proof}

\subsection{A useful lemma} 
For a field ${\mathbbm{k}}$ of characteristic $p > 0$, the Frobenius map 
${\mathbbm{k}} \rightarrow {\mathbbm{k}}; a \mapsto a^p$ is a 
${\mathbbm{k}}$-linear map.    Moreover,  it holds that
\begin{equation} \label{eq:contract power}
    u^p \circ w^p = (u \circ w)^p
\end{equation}
for any $u$ and $w$ homogeneous of the same degree. To see this, by the linearity of the
Frobenius map, both sides of Equation~(\ref{eq:contract power}) are bilinear. Hence, 
it suffices to check  the equality for $u,w$  monic monomials, and this
clearly holds.

For the proof of Theorem~\ref{thm:Parseval}, the following lemma is useful.

\begin{lemma} \label{lem:vanish}
   Assume the  field ${\mathbbm{k}}$ has characteristic $p > 0$.
    Let $g$ be a homogeneous polynomial in $R={\mathbbm{k}}[x_1,\ldots,x_m]$ of degree $d$, and 
    let $s$ be a non-negative integer with $s \geq d$.
 We denote by $(g)_s \subset R$ the $s$-th graded part of the homogeneous ideal  $(g)$.
    If $\phi:R_s \rightarrow {\mathbbm{k}}$ is a ${\mathbbm{k}}$-linear map with $(g)_s \subseteq \ker \phi$, then
    \begin{align} \label{eq:vanish}
        \sum_{u \in {\mathcal M}_s} ((x_1^{p-1}\cdots x_m^{p-1} u^p) \circ (g^p v)) (\phi (u))^p = 0
    \end{align}
   for all  homogeneous polynomials $v \in R$  of degree $m(p-1)+sp-dp$.
\end{lemma}
\begin{proof}
    It suffices to prove Equality~(\ref{eq:vanish}) when $v$ is a monic monomial. 
Let $v=x_1^{j_1}\cdots x_m^{j_m}$.
    
    Suppose that $j_k \not \equiv p-1$ (mod $p$) for some $k$, with $1 \leq k \leq m$.
    Write $g^p$ as a sum of  monomials. If a monomial $x_1^{l_1} \cdots x_m^{l_m}$ appears 
 as a  summand then all $l_1, \ldots, l_m$ must be multiples of $p$. Thus, $g^pv$ is supported 
on those monomials $x_1^{l_1} \cdots x_m^{l_m}$ with $l_t \equiv j_t$ (mod $p$) for all
$1 \leq t \leq m$.
    Thus, $(x_1^{p-1}\cdots x_m^{p-1} u^p) \circ (g^p v)=0$ for any $u \in {\mathcal M}_s$.
        
    Suppose now that $j_t \equiv p-1$ (mod $p$) for all $1 \leq t \leq m$.
    Then, we can write $v=x_1^{p-1}\cdots x_m^{p-1} {v_0}^p$ for some monic monomial $v_0$ of degree $s-d$. We have
    \begin{align*} 
        \sum_{u \in {\mathcal M}_s} ((x_1^{p-1}\cdots x_m^{p-1} u^p) \circ (g^p v)) (\phi (u))^p 
        =&\sum_{u \in {\mathcal M}_s} ((u^p) \circ (g^p v_0^p)) (\phi (u))^p & \notag \\ 
        =& \sum_{u \in {\mathcal M}_s} (u\circ (gv_0))^p (\phi(u))^p & \notag \\
        =& \left(\sum_{u \in {\mathcal M}_s} (u\circ (gv_0)) \phi(u)) \right)^p & \notag \\
        =& \left (\phi\left(\sum_{u \in {\mathcal M}_s} (u\circ (gv_0)) u \right) \right)^p \\
        =& \left (\phi (gv_0) \right)^p \\
       =&  0,
    \end{align*}
     where we used that         
     $\sum_{u \in {\mathcal M}_s} (u\circ w) u =  w$   for all $w \in R_s$   and 
    $\phi ( gv_0 ) = 0$  by the assumption  $(g)_s \subseteq \ker \phi$.  
\end{proof}

\subsection{Proof of Theorem~\ref{thm:Parseval}}
We are now ready to prove Theorem~\ref{thm:Parseval}
\begin{proof}[Proof of Theorem~\ref{thm:Parseval}]
    Since both sides of Equation~(\ref{eq:PI}) are ${\mathbbm{k}}$-linear with respect to $w$,
    it suffices to check Equation~(\ref{eq:PI}) on a spanning set of $R_s$.
    Since, by Proposition~\ref{prop:well-def}, the $1$-dimensional ${\mathbbm{k}}$-vector space 
 $(R/I)_s$ is spanned by $\pi(z_0)$, it suffices 
   to prove Equation~(\ref{eq:PI}) for the following two cases: (i)~ $w=g_i f$, where $f$ is a polynomial 
 of degree $s-\deg g_i$, and (ii) $w=z_0$.

    For case (i), without loss of generality, we may assume that $i=1$. 
   By applying Lemma~\ref{lem:vanish} for $v=g_2^{p-1}\cdots g_m^{p-1}f$ and $g=g_1$, we see that 
    the right-hand side of Equation~(\ref{eq:PI}) is $0$, and hence Equation~(\ref{eq:PI}) holds.

    For case (ii), by Lemma~\ref{lem:member}, we have that, for  all $i$ with $1 \leq i  \leq m$, there 
    exist  $\lambda_i \in R$  with $\deg \lambda_i = m(p-1)+sp-p\deg g_i$ such that
    \[
        g_1^{p-1}\cdots g_m^{p-1}z_0 = x_1^{p-1} \cdots x_m^{p-1} z_0^p + \sum_{i=1}^m \lambda_i g_i^p.
    \]
    Thus, the right-hand side of Equation~(\ref{eq:PI}) for $w=z_0$ is $A+\sum_{i=1}^m B_i$, where
    \begin{align*} 
        &A=\sum_{u \in {\mathcal M}_s} 
            ((x_1^{p-1}\cdots x_m^{p-1} u^p)\circ(x_1^{p-1}\cdots x_m^{p-1} z_0^p)) (\mathrm{vol}(\pi(u)))^p,
        \\
        &B_i=\sum_{u \in {\mathcal M}_s} ((x_1^{p-1}\cdots x_m^{p-1} u^p)\circ(\lambda_i g_i^p)) (\mathrm{vol}(\pi(u)))^p.
    \end{align*}
    We have 
    \begin{align*}  
     A=& \sum_{u \in {\mathcal M}_s} (u^p\circ z_0^p) (\mathrm{vol}(\pi(u)))^p =
        \left( \sum_{u \in {\mathcal M}_s} (u\circ z_0) \mathrm{vol}(\pi(u)) \right)^p  \\
      =&  \left(  (\mathrm{vol} \circ \pi) (\sum_{u \in {\mathcal M}_s}(u\circ z_0) u) \right)^p 
      = ((\mathrm{vol} \circ \pi) (z_0))^p=1, 
    \end{align*}
    where we used again  that $\; \sum_{u \in {\mathcal M}_s} (u\circ w) u = w \;$ for all 
    $w \in R_s$. 
    Moreover,  by  Lemma~\ref{lem:vanish} we have that      $B_i=0$ for all $1 \leq i \leq m$.
    Hence, when $w = z_0$  both sides of Equation~(\ref{eq:PI}) are equal to $1$.
\end{proof}

\section{Differential identities, $p$-anisotropy and 
              Lefschetz property for generic complete intersections}  \label{sec!dislp}

The present  section contains applications of Theorem~\ref{thm:Parseval}.  In different contexts, analogous
 applications of Parseval-Reyleigh identities   have appeared in  \cite{APP}  and  \cite{KLS}.

Let $\mathbb{N}=\{0,1,2,\ldots\}$ be the set of all nonnegative integers. 
For $\bm{r} =  (r_1, \dots , r_m)  \in \mathbb{N}^m$,  we set  $|\bm{r}|:=r_1+\cdots+r_m$.
For a nonnegative integer $k$,  we set  $\mathbb{N}^m_k:=\{\bm{r}\in\mathbb{N}^m:|\bm{r}|=k\}$.
\subsection{Generic complete intersection reduction}
\begin{defi}
Let ${\mathbbm{k}}$ be a field, and let $R={\mathbbm{k}}[x_1,\ldots,x_m]$ be the  polynomial ring
over ${\mathbbm{k}}$ in $m$ variables.
Let $d_1,\ldots,d_m$ be positive integers.
Let
\begin{equation}
    \mathcal{I}_0=\{(i,\bm{r}):i=1,\ldots,m,\, \bm{r}\in \mathbb{N}^m_{d_i}\},
\end{equation}
and let 
\[
{\mathbbm{K}}={\mathbbm{k}}(a_{i,\bm{r}}:(i,\bm{r}) \in\mathcal{I}_0),
\]
where $a_{i,\bm{r}}$ are new variables. Clearly ${\mathbbm{K}}$ is a purely
 transcendental field extension of ${\mathbbm{k}}$.  
In the polynomial ring ${\mathbbm{K}}[x_1,\ldots,x_m]$ we set
\[
g_i = \sum_{\bm{r} \in \mathbb{N}^m_{d_i}} a_{i,\bm{r}} \bm{x}^{\bm{r}} \qquad \text{ for } i=1,\ldots,m,
\]
where $\bm{x}^{\bm{r}}= \prod_{t=1}^m x_t^{r_t}$.
The quotient ring ${\mathcal A}={\mathbbm{K}}[x_1,\ldots,x_m]/(g_1,\ldots,g_m)$ is called 
the \emph{generic $(d_1,\ldots,d_m)$-complete intersection reduction} of $R$.    
\end{defi}
We call each $a_{i,\bm{r}}$ an \emph{auxiliary indeterminate}, and $\mathcal{I}_0$ the \emph{index set}.

\subsection{Differential identities} \label{sec:diff}
For  the remaining of the present  section we assume that ${\mathbbm{k}}$ is a field of positive characteristic $p$.
Let ${\mathcal A} = {\mathbbm{K}}[x_1,\ldots,x_m]/(g_1,\ldots,g_m)$ be the generic $(d_1,\ldots,d_m)$-complete 
intersection reduction of $R={\mathbbm{k}}[x_1,\ldots,x_m]$.
Let $s=-m + d_1+\cdots +d_m$ be the top degree of ${\mathcal A}$, and 
let $\mathrm{vol}:{\mathcal A}_s \rightarrow {\mathbbm{K}}$ be the residue map. 
Let $\pi:{\mathbbm{K}}[x_1,\ldots,x_m] \rightarrow {\mathcal A}$ be the natural projection.

By a multisubset of a set $S$, we mean a multiset consisting of elements of $S$.
\begin{defi}
    We say that a multisubset $\mathcal{I}$ of the index set $\mathcal{I}_0$  has the Property~($*$) if 
    \begin{quote}
        ($*$)  $\; \;$ For each $i=1,\ldots,m$, there are exactly $p-1$ elements of the form $(i,\bm{r})$ in $\mathcal{I}$.
    \end{quote}
\end{defi}

For $(i,\bm{r}) \in \mathcal{I}_0$, we denote by  
$\partial_{a_{i,\bm{r}}}: {\mathbbm{K}} \rightarrow {\mathbbm{K}}$ 
the differential operator   $\frac{\partial}{\partial{a_{i,\bm{r}}}}$. In other words,
$\partial_{a_{i,\bm{r}}}$ 
is the partial derivative with respect to the auxiliary indeterminate $a_{i,\bm{r}}$.
For a finite multisubset $\mathcal{I}=\{\!\{(i_1,\bm{r}_1),\ldots,(i_k,\bm{r}_k)\}\!\}$ of $\mathcal{I}_0$ 
we define the differential operator $\partial^\mathcal{I}$ by
\[
    \partial^\mathcal{I}:=\partial_{a_{i_1,\bm{r}_1}} \circ \cdots \circ \partial_{a_{i_k,\bm{r}_k}}
\]
and  the monic monomial $\bm{x}^\mathcal{I}$ by
\[
    \bm{x}^\mathcal{I} := \bm{x}^{\bm{r}_1} \cdots \bm{x}^{\bm{r}_k}.
\]
It is clear that  if $\mathcal{I}$ satisfies Property~($*$) then 
$\partial^\mathcal{I}$ is a differential operator of order $(p-1)m$ and 
$\bm{x}^\mathcal{I}$ is a monomial of degree $(p-1)(m+s) = (p-1) \sum_{i=1}^m \deg g_i$.

Moreover,  given $(i,\bm{r}) \in \mathcal{I}_0$ and a finite multisubset $\mathcal{I}$ of $\mathcal{I}_0$,
we denote by  $m_{\mathcal{I}, (i,\bm{r})}$ the number of appearances of $(i,\bm{r})$ in  $\mathcal{I}$,
and we set
\[
       \mathcal{I} \, ! = \prod_{(i,\bm{r}) \in \mathcal{I}_0}   (m_{\mathcal{I}, (i,\bm{r})})!
\]
where $0! =1$ and $n! = n ((n-1)!)$ when $n \geq 1$.

For a monic Laurent  monomial 
\[
      u  \in {\mathbbm{K}}[x_1,\ldots,x_m, x_1^{-1}, \ldots, x_m^{-1}],
\]
we denote by  $u^{1/p}$ the  (unique) monic  Laurent   monomial 
whose $p$-th power is equal to $u$ if such a monomial exists, and $0$ otherwise.
In other words, if $u=\prod_{i=1}^m x_i^{a_i}$ with $a_i \in \mathbb{Z}$,
then $u^{1/p} =\prod_{i=1}^m x_i^{a_i/p}$
if $p$ divides $a_i$ for all $i$ and $u^{1/p} =0 $ if there exists $i$ such that $p$ 
does not divide $a_i$.  Moreover, we denote by 
\[
   \phi :  {\mathbbm{K}}[x_1,\ldots,x_m, x_1^{-1}, \ldots, x_m^{-1}] \to 
              {\mathbbm{K}}[x_1,\ldots,x_m]
\]
the unique ${\mathbbm{K}}$-linear map such that
$ \phi (\bm{x}^{\bm{r}}) = \bm{x}^{\bm{r}}$ if $\bm{r} \in \mathbb{N}^m$ and $0$ otherwise. 
Hence, the restriction of $\phi$ to $ {\mathbbm{K}}[x_1,\ldots,x_m]$ is the identity map,
and the kernel of $\phi$ is spanned by all Laurent monomials that have at least
one negative exponent.

In  our present setting where  ${\mathcal A}$  is a  generic complete intersection reduction,   
we can transform  the  Parseval-Rayleigh identities of  Theorem~\ref{thm:Parseval} 
to differential identities  as follows. 
\begin{theorem}[\bf{Differential identities}] \label{thm:diff}	
    For $\bm{s} \in \mathbb{N}^m$ with $|\bm{s}|=s$ and a multisubset ${\mathcal I}$ of ${\mathcal I}_0$ with 
the Property~($*$), we have
    \begin{align} \label{eq:diff}
        \partial^{\mathcal{I}} ( (\mathrm{vol} \circ \pi)(\bm{x}^{\bm{s}})) =  (-1)^m
              \left(  (\mathrm{vol}\circ \pi \circ \phi )  
               \left(\left(\bm{x}^{\bm{s}}\bm{x}^{{\mathcal I}}\bm{x}^{(1-p)\mathbf{1}}\right)^{1/p}\right) \right)^p.
    \end{align}
    Here, $\mathbf{1}=(1,\ldots,1) \in \mathbb{N}^m$ is the vector with all coordinates $1$.
\end{theorem}
\begin{proof}
    By the Parseval-Rayleigh identities of   Theorem~\ref{thm:Parseval} and the identity $(fg^p)' = f'g^p$, we get
    \begin{align} \label{eq:diff sum}
        \partial^{{\mathcal I}} ((\mathrm{vol} \circ \pi) (\bm{x}^{\bm{s}})) &= 
           \sum_{u \in {\mathcal M}_s} 
    \partial^{{\mathcal I}} \left( (\bm{x}^{(p-1)\bm{1}}u^p)\circ(g_1^{p-1}\cdots g_m^{p-1} \bm{x}^{\bm{s}}) \right) 
             (\mathrm{vol} ( \pi (u)))^p.
    \end{align}
    Combining Wilson's theorem with the multinomial theorem we get 
    \begin{equation}     \label{eqn!multinomialthm}
        g_1^{p-1}\cdots g_m^{p-1} = (-1)^m    \sum_{\mathcal I}   \left(  \frac{1}  {  {\mathcal I} \, ! }  \,
                           x^{\mathcal{I}} \,  \prod_{(i,\bm{r}) \in \mathcal{I}} a_{i,\bm{r} }  \, \right),
    \end{equation}
    where the sum is taken over all multisubsets ${\mathcal I}$ of ${\mathcal I}_0$ with the Property~($*$). 
    For multisubsets ${\mathcal I}, {\mathcal I}'$ of ${\mathcal I}_0$ of the same length,  it is easy to see that 
    \begin{equation*}
        \partial^{\mathcal I} \left( \prod_{(i,\bm{r}) \in \mathcal{I}'} a_{i,\bm{r}} \right) =
                 (\mathcal{I} \, ! )  \;  \delta_{{\mathcal I},{\mathcal I'}}
    \end{equation*}
    where $\delta$ denotes the Kronecker delta. Thus, for a monomial $u \in {\mathcal M}_s$ we have
    \begin{align} \label{eq:termwise deriv}
        \partial^{\mathcal I} \left((x^{(p-1)\mathbf{1}} u^p) \circ (g_1^{p-1}\cdots g_m^{p-1} x^{\bm{s}})\right)=
            \begin{cases}
                (-1)^m  &  \text{ if } \bm{x}^{(p-1)\mathbf{1}} u^p=\bm{x}^{\mathcal I}\bm{x}^{\bm{s}}, \\
                0  &  \text{ otherwise}.
            \end{cases}
    \end{align}  
    Equations~(\ref{eq:diff sum}) and (\ref{eq:termwise deriv}) imply the desired 
    Equation~(\ref{eq:diff}).
\end{proof}  

\subsection{$p$-anisotropy, Lefschetz property}
Suppose that we are in the same setting as in Section~\ref{sec:diff}.
Let $\ell=x_1+\cdots+x_m \in {\mathbbm{K}}[x_1,\ldots,x_m]$.
\begin{prop} \label{prop:anisSLP}
    For $\alpha \in R_i$ and $k \in \mathbb{N}$ with $pi+k \leq s$, we have that
    if $\pi(\alpha) \neq 0$  then $\pi(\alpha^p\ell^k) \neq 0$.
\end{prop}
We first prove the following special case of Proposition~\ref{prop:anisSLP}.
\begin{lemma} \label{lem:anisWLP}
     For $\alpha \in R_i$ and $k \in \mathbb{N}$ with $k \leq p-1$ and $pi+k \leq s$, we have 
  that if $\pi(\alpha) \neq 0$  then $\pi(\alpha^p\ell^k) \neq 0$.
\end{lemma}
\begin{proof}
    Write $\alpha$ as $\alpha=\sum_{\bm{t} \in \mathbb{N}^m_i} \lambda_{\bm{t}} x^{\bm{t}}$ 
         for some $\lambda_{\bm{t}} \in {\mathbbm{K}}$.
    For any monomial $\bm{x}^{\bm{s}}$ with degree $|\bm{s}|=s-pi$ and any
   multisubset ${\mathcal I}$ of ${\mathcal I}_0$ with the Property~($*$)     we have
    \begin{align}
         \partial^{{\mathcal I}} ( (\mathrm{vol} \circ \pi)(\alpha^p \bm{x}^{\bm{s}}))
          = &\sum_{\bm{t} \in \mathbb{N}^m_i} \partial^{{\mathcal I}} 
               ( (\mathrm{vol}\circ \pi)(\lambda_{\bm{t}}^p \bm{x}^{p\bm{t}+\bm{s}})) \notag\\
        =&\sum_{\bm{t} \in \mathbb{N}^m_i} \lambda_{\bm{t}}^p  \partial^{{\mathcal I}}
                   ( (\mathrm{vol} \circ \pi)( \bm{x}^{p\bm{t}+\bm{s}})) & \notag\\
        =& (-1)^m \sum_{\bm{t} \in \mathbb{N}^m_i} \lambda_{\bm{t}}^p 
              \left( (\mathrm{vol} \circ \pi \circ \phi)\left( \bm{x}^{p\bm{t}}\left(\bm{x}^{\bm{s}}
                    \bm{x}^{\mathcal I}\bm{x}^{(1-p)\bm{1}}\right) \right)^{1/p} \right)^p 
                   & (\text{by Theorem~\ref{thm:diff}}) \notag\\
        =& (-1)^m \sum_{\bm{t} \in \mathbb{N}^m_i} \lambda_{\bm{t}}^p
            \left( (\mathrm{vol} \circ \pi \circ \phi )\left(  
                  \bm{x}^{\bm{t}}\left(\bm{x}^{\bm{s}} 
                    \bm{x}^{\mathcal I}\bm{x}^{(1-p)\bm{1}}\right)^{1/p}  \right)\right)^p \notag\\
        =& (-1)^m \left( (\mathrm{vol} \circ \pi \circ \phi)\left( \sum_{\bm{t} \in \mathbb{N}^m_i} 
                 \lambda_{\bm{t}} \bm{x}^{\bm{t}}\left(\bm{x}^{\bm{s}} 
                    \bm{x}^{\mathcal I}\bm{x}^{(1-p)\bm{1}}\right)^{1/p}  \right)\right)^p \notag\\
        =&(-1)^m    \left( (\mathrm{vol}\circ \pi \circ \phi) \left(
                   \alpha\left(\bm{x}^{\bm{s}}\bm{x}^{{\mathcal I}}
               \bm{x}^{(1-p)\mathbf{1}}\right)^{1/p} \right) \right)^p.  \label{eq:diff with pth power}
      \end{align}
   By the multinomial theorem, we have  
   \[
         \ell^k =  \sum_{\bm{r} \in \mathbb{N}^m_k} \mu_{\bm{r}}  \bm{x}^{\bm{r}},   \quad  \quad
            \text{where}  \quad \quad      \mu_{\bm{r}} =  \frac{k!}{\prod_{t=1}^m  r_t ! }
   \]
for all $\bm{r} = (r_1, \dots ,r_m)  \in \mathbb{N}^m_k$.
Using  Equation~(\ref{eq:diff with pth power}),  we get that 
    for any monomial $\bm{x}^{\bm{s}}$ with degree $|\bm{s}|=s-pi-k$ 
   and any
   multisubset ${\mathcal I}$ of ${\mathcal I}_0$ with the Property~($*$)     we have
    \begin{align} \label{eq:diff with ell}
        \partial^{\mathcal I} ( (\mathrm{vol}\circ \pi)(\alpha^p\ell^k\bm{x}^{\bm{s}}) )
               &=  (-1)^m   \sum_{\bm{r} \in \mathbb{N}^m_k} \mu_{\bm{r}}^{p} \left( (\mathrm{vol}\circ \pi \circ \phi) 
                       \left( \alpha\left(\bm{x}^{\bm{r}+\bm{s}}\bm{x}^{{\mathcal I}}
                   \bm{x}^{(1-p)\mathbf{1}}\right)^{1/p}\right) \right)^p.
    \end{align}
    Note that the summands of the right-hand side of Equation~(\ref{eq:diff with ell}) are zero except
    possibly  for one summand.
    To see this, observe that if the summand corresponding to $\bm{r}$ is nonzero, then the exponent vector of 
             $\bm{x}^{\bm{r}+\bm{s}}\bm{x}^{{\mathcal I}}\bm{x}^{(1-p)\mathbf{1}}$ must be in $(p\mathbb{Z})^m$. 
    But then for any other $\bm{r}' \in \mathbb{N}^{m}_{k}$ there exists a coordinate $i$ 
           with $r_i \neq r_i'$. Since $1 \leq |r_i - r_i'| \leq k <p$, it follows that
         the exponent vector of $\bm{x}^{\bm{r'}+\bm{s}}\bm{x}^{{\mathcal I}}\bm{x}^{(1-p)\mathbf{1}}$ 
               is not in $(p\mathbb{Z})^m$.  

   Since $\mathcal{A}$ is Artinian Gorenstein it satisfies Poincar\'e Duality.
    Therefore,   $\pi(\alpha) \neq 0$ implies that there exists  a monomial $\bm{x}^{\bm{u}} \in R$ 
    of degree $|\bm{u}|=s-i$ 
     such that $\pi(\alpha\bm{x}^{\bm{u}}) \neq 0$.
    We can pick $\bm{r_0} \in \mathbb{N}^m_k$, $\bm{s_0} \in \mathbb{N}^m_{s-pi-k}$ and a
   multisubset ${\mathcal I}$ of ${\mathcal I}_0$ with the Property~($*$)   
   so that
    \[
        \bm{x}^{p\bm{u}}=\bm{x}^{\bm{r_0}+\bm{s_0}}\bm{x}^{{\mathcal I}}\bm{x}^{(1-p)\mathbf{1}}
    \]
    holds. By the uniqueness of the nonzero term on 
   the right-hand side of Equation~(\ref{eq:diff with ell}), we get 
    \[
    \partial^{\mathcal I} ((\mathrm{vol}\circ \pi) (\alpha^p\ell^k\bm{x}^{\bm{s}_0}) ) = 
         (-1)^m \mu_{\bm{r_0}}^p   ((\mathrm{vol}\circ \pi ) (\alpha \bm{x}^{\bm{u}}))^p \neq 0,
    \]
    due to the fact that since  $k < p$ the element  $\mu_{\bm{r_0}}^p$ of $ {\mathbbm{K}}$ is nonzero.
    This implies that $\alpha^p\ell^k\neq 0$.
\end{proof}

\begin{proof}[Proof of Proposition~\ref{prop:anisSLP}]
   We assume that Proposition~\ref{prop:anisSLP} is  not true. We denote by $k_0$ the least value of 
    $k \in \mathbb{N}  $ such that the proposition is not true. We 
   perform   Euclidean division  $k_0 =  j p +r$, with $j,r \in \mathbb{Z}$ and  $0 \leq r \leq p-1$.  
   We have $k_0 > j$.  So,  $\pi(\alpha^p\ell^j) \neq 0$.  This, implies that  $\pi(\alpha \ell^j) \neq 0$. Applying Lemma~\ref{lem:anisWLP}  for  $\alpha = \alpha \ell^{j} $ and $k=r$ we get
    $\pi(\alpha^p\ell^{k_0}  ) \neq 0$, which contradicts the definition of $k_0$.
\end{proof}
 
\begin{remark}  
    In particular, Proposition~\ref{prop:anisSLP} implies that ${\mathcal A}$ satisfies the $p$-anisotropy:
     If  $0 \leq i \leq \frac{s}{p}$ and 
     $\alpha \in {\mathcal A}_i$  is nonzero, then $\alpha^p$ is nonzero.
\end{remark}
 
Another corollary is the following result regarding the Lefschetz property.
\begin{corollary} \label{cor:Lefschetz}
    Assume $0 \leq i \leq \frac{s}{p}$. Then the multiplication map 
    $\times\pi(\ell)^{s-pi}:{\mathcal A}_i \rightarrow {\mathcal A}_{i+s-pi}$ is injective.
    In particular, if $p=2$ then ${\mathcal A}$ has the Strong Lefschetz Property.
\end{corollary}
\begin{proof}
    If $\alpha \in {\mathcal A}_i$ is nonzero  then,  by 
     Proposition~\ref{prop:anisSLP},  $\alpha^p\pi(\ell)^{s-pi}$ is nonzero. As a consequence, 
       $\alpha \pi(\ell)^{s-pi}$ is nonzero.
\end{proof}
 \bibliographystyle{plain}
\bibliography{myreference}

\end{document}